\documentclass[12pt]{amsart}
\usepackage[english]{babel}
\usepackage[utf8]{inputenc}
\usepackage{lipsum}
\usepackage{mathrsfs}
\usepackage{stix}
\usepackage{fullpage}
\usepackage{mathtools}
\usepackage{amsmath}
\usepackage{tikz-cd}
\usepackage[colorlinks]{hyperref}
\usepackage{fullpage}
\usepackage{mathtools}
\usepackage{amsmath}
\usepackage{tikz-cd}
\numberwithin{equation}{section}
\theoremstyle{plain}
\newtheorem{theorem}{Theorem}[section]

\newtheorem{rem}[theorem]{Remark}

\def\X{\mathbb{Y}}
\allowdisplaybreaks
\begin{document}
\title[Weighted Levin-Cochran-Lee type inequalities   ]{Anisotropic Weighted Levin-Cochran-Lee type inequalities on homogeneous Lie groups}

\author[Michael Ruzhansky]{Michael Ruzhansky}
\address{
	Michael Ruzhansky:
	\endgraf
	Department of Mathematics: Analysis, Logic and Discrete Mathematics
	\endgraf
	Ghent University, Belgium
	\endgraf
	and
	\endgraf
	School of Mathematical Sciences
	\endgraf
	Queen Mary University of London
	\endgraf
	United Kingdom
	\endgraf
	{\it E-mail-} {\rm michael.ruzhansky@ugent.be}
}

\author[A. Shriwastawa]{Anjali Shriwastawa}
\address{
 Anjali Shriwastawa:
 \endgraf
  DST-Centre for Interdisciplinary Mathematical Sciences  \endgraf
  Banaras Hindu University, Varanasi-221005, India
  \endgraf
  {\it E-mail-} {\rm anjalisrivastava7077@gmail.com}}

\author[B.Tiwari]{Bankteshwar Tiwari}
\address{
 Bankteshwar Tiwari :
 \endgraf
  DST-Centre for Interdisciplinary Mathematical Sciences  \endgraf
  Banaras Hindu University, Varanasi-221005, India
  \endgraf
  {\it E-mail-} {\rm banktesht@gmail.com}}


\subjclass[2010]{26D10, 22E30, 45J05.}
\keywords{ Integral Hardy inequalities, Anisotropic Levin-Cochran-Lee inequalities, Homogeneous Lie groups, Quasi-norm, Weighted exponential inequalities}

\begin{abstract} In this paper, we first prove the weighted  Levin-Cochran-Lee type inequalities on homogeneous Lie groups for arbitrary weights, quasi-norms, and $L^p$-and $L^q$-norms. Then,  we derive a sharp weighted inequality involving specific  weights given in the form of quasi-balls in homogeneous Lie groups. Finally, we also calculate the sharp constants for the aforementioned  inequalities.
\end{abstract}
\maketitle
\allowdisplaybreaks
\section{History and Introduction} \label{Intro}

 In 1984, Cochran and Lee  rediscovered an  exponential weighted inequality in their  paper \cite{CL}, which was proved earlier in an unnoticed paper of V. Levin \cite{L1} in 1938 written in the Russsian language. We recall the following exponential weighted inequalities proved in the papers of Levin \cite{L1} and Cochran and Lee  \cite{CL}.  
\begin{theorem}
 Let $\epsilon$ and $a$ be two real numbers. Suppose that $f$ is a positive function such that the function $t^{\epsilon-1}\log f(t)$ is locally integrable on $(0,\infty).$ Then the inequality
\begin{align}\label{Levin2}
    \int_0^\infty\left[\exp\left(\epsilon\,x^{-\epsilon}\int_0^x t^{\epsilon-1}\log\,f(t)\,dt\right)\right]x^{a}\,dx\le
    \left(\exp\frac{a+1}{\epsilon}\right)\int_0^\infty x^a\,f(x)\,dx
\end{align} holds for $\epsilon>0,$ and 
\begin{align}\label{Levin1}
  \int_0^\infty\left[\exp\left(-\epsilon\,x^{-\epsilon}\int_x^\infty t^{\epsilon-1}\log\,f(t)\,dt\right)\right]x^{a}\,dx\le
    \left(\exp\frac{a+1}{\epsilon}\right)\int_0^\infty x^a\,f(x)\,dx  
\end{align} holds for $\epsilon>0.$ Moreover, the constant $\exp\left(\frac{a}{\epsilon}\right)$ is the best possible constant. 
\end{theorem}

The inequality \eqref{Levin2} is called the Levin-Cochran-Lee type
inequality and its complementary inequality \eqref{Levin1} was proved by E. R. Love in \cite{L2}, which was again
reproved by G.-S. Yang and Y. J. Lin \cite{YL}. It is worth noting that  inequality \eqref{Levin2} is a generalization of the famous Knopp inequality \cite{K}, which can be obtained by setting $a=0$ and $\epsilon=1$ in \eqref{Levin2}.
Thereafter, several works have been devoted to the study of these exponential type inequalities in different forms and in
different settings such as higher dimensional Euclidean spaces and Euclidean balls, by many authors.  It is clearly impossible to give a complete
overview of the available literature, therefore we refer to the books, surveys and papers \cite{CG, CL,CPP,DHK,JPW,K,KPS} and references
therein.

 Čižmešija et al. \cite{CPP} investigated an $n$- dimensional analogue of \eqref{Levin2} by replacing the intervals $(0,\infty)$ by $\mathbb{R}^n$ and the means are considered over the balls in $\mathbb{R}^n$ centered as origin. We state this inequality as follows:
 \begin{theorem}
 Let $f$ be a positive function on $\mathbb{R}^n$ and let $\mathbb{B}(0,|x|)$ be the ball in $\mathbb{R}^n$ with radius $|x|,\, x\in \mathbb{R}^n,$ centered at the origin, with its volume (with respect to the Lebesgue measure on $\mathbb{R}^n$) denoted by $|\mathbb{B}(0,|x|)|$. Then we have the following inequality
\begin{align}\label{INQ12}
\nonumber\int_{\mathbb{R}^n}\Bigg[\exp\Bigg(\epsilon\,{|\mathbb{B}(0,|x|)|}^{-\epsilon}\int_{\mathbb{B}(0,|x|)} {|\mathbb{B}(0,|y|)|}^{\epsilon-1}\log\,f(y)\,dy&\Bigg)\Bigg]|\mathbb{B}(0,|x|)|^a\,dx\\& \leq 
   \left( \exp{\frac{a+1}{\epsilon}}\right)\int_{\mathbb{R}^n} f(x)\,|\mathbb{B}(0,|x|)|^adx,    
\end{align}
where $a$ and $\epsilon>0$ are two real numbers. Moreover,  the constant  $\exp{\frac{a+1}{\epsilon}}$ appearing in  \eqref{INQ12} is a sharp constant. 
 \end{theorem}
 
 The inequality \eqref{INQ12} was further generalized by Jain et al. \cite{JPW} to a more general situation involving general weight functions on the Euclidean space. 
 
 The main objective of this paper is to prove a Levin-Cochran-Lee type inequality involving general weight functions on homogeneous (Lie) groups equipped with a quasi-norm
$|\cdot|$ and  a family of
dilations compatible with the group law. For a detailed description of analysis on homogeneous groups, we refer to \cite{FS,FR, RS}.
Particular examples of homogeneous groups are the Euclidean space $\mathbb{R}^n$ (in which case $Q = n$), the
Heisenberg group, as well as general stratified groups (homogeneous Carnot groups) and graded groups. Recently, Hardy type inequalities and their best constants have been extensively investigated in non-commutative settings (e.g. Heisenberg groups, graded groups, homogeneous groups); we cite \cite{FL12,RS, RY, RS1, RST1} just to a mention a few of them.

Recently, the first author and Verma \cite{RV} obtained several characterizations of weights for two-weight integral Hardy inequalities to hold on general metric measure spaces possessing polar decompositions for the range $1<p \leq q<\infty$ (see, \cite{RV1} for the case $0<q<p$ and $1<p<\infty$). Using this, one deduced the weighted integral Hardy inequality on homogeneous groups, hyperbolic spaces and Cartan-Hadamard manifolds. In particular, one proved the following theorem  \cite{RV} which will be useful to establish  results of the present paper. 
\begin{theorem}\label{TH41} Let $\mathbb{G}$ be a homogeneous group with the homogeneous dimension $Q,$ and let $1<p\le q<\infty.$
Suppose that $u$ and $v$ are two weight functions on $\mathbb{G}.$ Then the inequality 
\begin{align}\label{Th1}
    \left(\int_{\mathbb{G}} \left(\int_{\mathbb{B}(0,|x|)} f(y)\,dy\right)^q u(x)\,dx \right)^\frac{1}{q}\le C\left(\int_{\mathbb{G}}f^p(x)\,v(x) \,dx\right)^\frac{1}{p}
\end{align}
holds for all non-negative functions $f$ on $\mathbb{G}$ if and only if 
\begin{align*} 
   \textbf{A}_Q:=\sup_{x\in\mathbb{G}}\left(\int_{\mathbb{G}\backslash\mathbb{B}(0,|x|)}u(y)\,dy\right) ^\frac{1}{q}\left(\int_{\mathbb{B}(0,|x|)}v^\frac{1}{1-p}(y)\, dy\right)^\frac{p-1}{p}<\infty,
\end{align*}
and the best constant $C$ in
\eqref{Th1} can be estimated in the following way:
\begin{align*}
   \textbf{A}_Q\le C\le \textbf{A}_Q \left(\frac{p}{p-1}\right)^\frac{p-1}{p}p^\frac{1}{q}.
\end{align*}

\end{theorem}

Very recently,  we have proved  a sharp version of Theorem \ref{TH41} in \cite{RST}. In fact, we have also calculated the precise value of sharp constants in respective inequalities on homogeneous groups. Using Theorem \ref{TH41} we prove the following result which is one of the main results of this paper. 

\begin{theorem}\label{Th2int}
Let $\mathbb{G}$ be a homogeneous group with the homogeneous dimension  $Q$ equipped with a quasi norm $|\cdot|$ and 
let $0<p\leq q<\infty.$  Suppose that $u$ and $v$ are two positive weight functions on $\mathbb{G}.$ Then, there exists a positive constant $C$ such that, for all positive functions $f$ on $\mathbb{G}$, the following inequality holds 
\begin{equation} \label{mainin}
    \left(\int_{\mathbb{G}} \left[ \exp \left( \frac{1}{|\mathbb{B}(0, |x|)|} \int_{\mathbb{B}(0, |x|)} \log f(y)\, dy \right) \right]^q u(x)\, dx \right) ^\frac{1}{q}\leq C \left(\int_{\mathbb{G}} f^p(x)\, v(x)\,dx \right)^{\frac{1}{p}},
\end{equation}
provided that 
\begin{equation}\label{coeff33in}
    D_Q:= \sup_{x \in \mathbb{G}} |\mathbb{B}(0, |x|)|^{\frac{1}{q}-\frac{1}{p}} u_1^{\frac{1}{q}}(x) \left[ \exp \left( \frac{1}{|\mathbb{B}(0, |x|)|} \int_{\mathbb{B}(0, |x|)} \log \frac{1}{v(y)}\,dy \right)\right]^{\frac{1}{p}}<\infty.
\end{equation} Here $u_1$ is the spherical average of $u,$ given by 
\begin{align}\label{U1in}
u_1(x):=\frac{1}{|\mathfrak{S}|} \int_{\mathfrak{S}} u(|x|\sigma)\, d\sigma,
\end{align} where $\mathfrak{S}= \{x \in \mathbb{G}: |x|=1\} \subset \mathbb{G}$ is the unit sphere with respect to the quasi-norm $|\cdot|$.

Moreover, the optimal constant $C$ in \eqref{main} can be estimated as follows:
\begin{equation}\label{DQin}
     0<C \leq \left(\frac{p}{q} \right)^{\frac{1}{q}} e^{\frac{1}{p}} D_Q.
\end{equation}
\end{theorem}

We will also prove a conjugate version (see Theorem \ref{congthm}) of Theorem \ref{Th2int}. Furthermore, we establish some stronger exponential inequalities on  the quasi-balls on homogeneous Lie groups (see Theorem \ref{THM2} and Theorem \ref{thm3.5}). 
In fact, we will prove the following result:
\begin{theorem} Let $\mathbb{G}$ be a homogeneous group with the homogeneous dimension  $Q$ equipped with a  quasi norm $|\cdot|.$
Let $0<p\le q<\infty$ and $a,b\in \mathbb{R}.$ Then for any $\epsilon>0$ and  for any arbitrary positive function $f$ on the homogeneous Lie group $\mathbb{G},$ the following inequality
\begin{align}\label{theorem}
   \nonumber\Bigg(\int_\mathbb{G}\Bigg[\exp\Bigg(\epsilon|\mathbb{B}(0,|x|)|^{-\epsilon}\int_{\mathbb{B}(0,|x|)}&|\mathbb{B}(0,|y|)|^{\epsilon-1}\log\,f(y) dy\Bigg)\Bigg]^q |\mathbb{B}(0,|x|)|^a dx\Bigg)^\frac{1}{q}\\&\le C\Bigg(\int_\mathbb{G} f^p(x)|\mathbb{B}(0,|x|)|^b\, dx\Bigg)^\frac{1}{p}
\end{align}holds for a positive finite constant $C$ if and only if \begin{align}
    p\left(a+1\right)-q\left(b+1\right)=0.
\end{align}
Moreover, the best constant $C$ in \eqref{theorem} satisfies 
\begin{align}
  \left(\frac{p}{q}\right)^\frac{1}{q} \epsilon^{\frac{1}{p}-\frac{1}{q}}\, \exp\bigg({\frac{b+1}{\epsilon p}-\frac{1}{p}}\bigg) \le C \le\left(\frac{p}{q}\right)^\frac{1}{q}\epsilon^{\frac{1}{p}-\frac{1}{q}}\, \exp\bigg({\frac{b+1}{\epsilon p}}\bigg).
\end{align}
\end{theorem}

For the proof, we follow the method developed in \cite{CPP, JPW} in the (isotropic and abelian) setting of  Euclidean spaces. We note that also in the abelian (both isotropic and anisotropic) cases of $\mathbb{R}^n$, our results provide new insights in view of the arbitrariness of the quasi-norm $|\cdot|$ which does not necessarily have to be the Euclidean norm.

Apart from Section \ref{Intro}, this manuscript is divided in two sections. In the next section, we will recall the basics of homogeneous Lie groups and some other useful concepts. The last section is devoted to presenting proofs of the main results of this paper.

Throughout this paper, the symbol $A\asymp B$ means $\exists\,C_{1},C_{2}>0$ such that $C_{1}A\leq B\leq C_{2}A$.

\section{Preliminaries: Basics on homogeneous Lie groups} \label{preli}

In this section, we recall the basics of homogeneous groups. For more details on homogeneous groups as well as several functional inequalities on homogeneous groups, we refer to  monographs \cite{FR, FS,RS} and references therein.

A Lie group $\mathbb{G}$ (identified with $(\mathbb{R}^N, \circ)$) is called a homogeneous group if it is equipped  with the dilation mapping $$ D_\lambda:\mathbb{R}^N \rightarrow \mathbb{R}^N,\quad\lambda >0,$$ defined as 
\begin{equation}
 D_\lambda(x)= (\lambda^{v_1}x_1,\lambda^{v_2}x_2,\ldots,\lambda^{v_N}x_N) ,\quad v_1,v_2,\dots,v_N > 0,
\end{equation} 
which is an automorphism of the group $\mathbb{G}$ for each $\lambda >0 $. At times, we will denote the image of $x\in \mathbb{G}$ under $D_\lambda$ by $\lambda(x)$ or, simply $\lambda x$.
The homogeneous dimension $Q$ of the  homogeneous group $\mathbb{G}$ is defined by 
$$Q=v_1+v_2+\dots+v_N.$$
It is well known that a homogeneous group is necessarily nilpotent and unimodular. The Haar measure $dx$ on $\mathbb{G}$ is nothing but the Lebesgue measure on $\mathbb{R}^N$. 

Let us denote  the volume of a measurable set $\omega \subset \mathbb{G}$ by $|\omega|$. Then we have the following consequences: for $\lambda >0$
\begin{equation}
    |D_\lambda(\omega)|=\lambda^Q |\omega|  \quad \text{and} \quad \int_{\mathbb{G}} f(\lambda x) dx = \lambda^{-Q}\int_{\mathbb{G}} f(x) dx.
\end{equation}
A quasi-norm on $\mathbb{G}$ is any continuous non-negative function $ |\cdot|:\mathbb{G} \rightarrow [0,\infty)$ satisfying the following conditions:
\begin{itemize}
    \item[(i)] $|x|=|x^{-1}|$ for all $x \in \mathbb{G},$
    \item[(ii)] $|\lambda x|=\lambda |x|$\, for all \, $x \in \mathbb{G}$ and $\lambda >0,$
    \item[(iii)] $|x|=0 \iff x=0.$
\end{itemize}

If $\mathfrak{S}= \{x \in \mathbb{G}: |x|=1\} \subset \mathbb{G}$ is the unit sphere with respect to the quasi-norm $|\cdot|$, then there is a unique Radon measure $\sigma$ on $\mathfrak{S}$ such that for all $f \in L^1(\mathbb{G})$, we have the following polar decomposition
\begin{equation} \label{polar}
    \int_{\mathbb{G}} f(x) dx =\int_0^\infty \int_\mathfrak{S} f(ry) r^{Q-1}d\sigma(y) dr.
\end{equation}

Firstly, let us  consider metric spaces $\mathbb X$ with a Borel measure $dx$ allowing for the following {\em polar decomposition} at $a\in{\mathbb Y}$: we assume that there is a locally integrable function $\lambda \in L^1_{loc}$  such that for all $f\in L^1(\mathbb X)$ we have
   \begin{equation}\label{EQ:polarintro}
   \int_{\mathbb Y}f(x)dx= \int_0^{\infty}\int_{\Sigma_r} f(r,\omega) \lambda(r,\omega) d\omega_{r} dr,
   \end{equation}
    for the set $\Sigma_r=\{x\in\mathbb{Y}:d(x,a)=r\}\subset \mathbb Y$ with a measure on it denoted by $d\omega_r$, and $(r,\omega)\rightarrow a $ as $r\rightarrow0$.
    
The examples of such metric measure spaces are the Euclidean space, homogeneous Lie groups, hyperbolic spaces and, more generally, Cartan-Hadamard manifolds. 

 Here we fix some notation which be used in the sequel. The letters $u$ and $v$ will be always used to denote the weights on homogeneous groups $\mathbb{G}.$ A quasi-ball in the homogeneous group $\mathbb{G}$ with radius $|x|,$ \, $x \in \mathbb{G},$ and centred at the origin will be denoted by  $\mathbb{B}(0,|x|)$. We denote the Haar measure of the unit sphere $\mathfrak{S}$ in $\mathbb{G}$   by $|\mathfrak{S}|.$ The Haar measure of the unit quasi-ball $\mathbb{B}(0,|x|),$ denoted by $|\mathbb{B}(0,|x|)|,$ can be calculated by using \eqref{polar} as
\begin{align}\label{measure}
 \nonumber|\mathbb{B}(0,|x|)|&=\int_{\mathbb{B}(0,|x|)} dy=\int_0^{|x|} r^{Q-1}\left(\int_\mathfrak{S}   d\sigma\right)\,dr \\&=\int_\mathfrak{S}\left(\int_0^{|x|} r^{Q-1} dr\right) d\sigma=\frac{|x|^Q|\mathfrak{S}|}{Q}.
\end{align}

For a given function $u$ on $\mathbb{G},$  the spherical average $u_1$  of $u$ is defined by 
\begin{align}
u_1(x):=\frac{1}{|\mathfrak{S}|} \int_{\mathfrak{S}} u(|x|\sigma)\, d\sigma,
\end{align} where $\mathfrak{S}= \{x \in \mathbb{G}: |x|=1\} \subset \mathbb{G}$ is the unit sphere with respect to the quasi-norm $|\cdot|$.

\section{Main Results}
In this section, we prove  the weighted Levin-Cochran-Lee type inequalities on a homogeneous Lie group equipped with a quasi-norm for arbitrary weights. We will derive sharp weighted inequalities on quasi-balls in homogeneous (Lie) groups involving specific weights and also calculate the sharp constant for these inequalities. 

We begin this section with the following Levin-Cochran-Lee inequality on homogeneous groups involving general weights.
\begin{theorem}\label{Th2}
Let $\mathbb{G}$ be a homogeneous group with the homogeneous dimension  $Q$ equipped with a quasi-norm $|\cdot|,$ and 
let $0<p\leq q<\infty.$  Suppose that $u$ and $v$ are two non-negative functions on $\mathbb{G}.$ Then, there exists a positive constant $C$ such that, for all positive functions $f$ on $\mathbb{G}$, the following inequality holds 
\begin{equation} \label{main}
    \left(\int_{\mathbb{G}} \left[ \exp \left( \frac{1}{|\mathbb{B}(0, |x|)|} \int_{\mathbb{B}(0, |x|)} \log f(y)\, dy \right) \right]^q u(x)\, dx \right) ^\frac{1}{q}\leq C \left(\int_{\mathbb{G}} f^p(x)\, v(x)\,dx \right)^{\frac{1}{p}},
\end{equation}
provided that 
\begin{equation}\label{coeff33}
    D_Q:= \sup_{x \in \mathbb{G}} |\mathbb{B}(0, |x|)|^{\frac{1}{q}-\frac{1}{p}} u_1^{\frac{1}{q}}(x) \left[ \exp \left( \frac{1}{|\mathbb{B}(0, |x|)|} \int_{\mathbb{B}(0, |x|)} \log \frac{1}{v(y)}\,dy \right)\right]^{\frac{1}{p}}<\infty.
\end{equation} Here $u_1$ is the spherical average of $u,$ given by
\begin{align}\label{U1}
u_1(x):=\frac{1}{|\mathfrak{S}|} \int_{\mathfrak{S}} u(|x|\sigma)\, d\sigma,
\end{align} where $\mathfrak{S}= \{x \in \mathbb{G}: |x|=1\} \subset \mathbb{G}$ is the unit sphere with respect to the quasi-norm $|\cdot|$.

Moreover, the optimal constant $C$ in \eqref{main} can be estimated as follows:
\begin{equation}\label{DQ}
    0< C \leq \left(\frac{p}{q} \right)^{\frac{1}{q}} \exp\bigg({\frac{1}{p}}\bigg) D_Q.
\end{equation}
\end{theorem}

\begin{proof}
 We  begin with the proof by rewriting Theorem \ref{TH41} by replacing $\frac{p}{\alpha},\,\frac{q}{\alpha},\,u(x)|\mathbb{B}(0,|x|)|^\frac{-q}{\alpha}$ and $f^\alpha$ in the places of $p,\,q,\,u(x)$ and $f,$ respectively, where $0<\alpha<p$. Indeed, we get the following  inequality
 \begin{align*}
   \left(\int_\mathbb{G}\left(\int_{\mathbb{B}(0,|x|)}f^\alpha (y)\,dy\right)^\frac{q}{\alpha}u(x)|\mathbb{B}(0,|x|)|^\frac{-q}{\alpha}\, dx\right)^\frac{\alpha}{q} \le C_\alpha\left(\int_\mathbb{G}f^{\alpha\frac{p}{\alpha}}(x)\,v(x)\,dx\right)^\frac{\alpha}{p}, 
 \end{align*}
which in turn implies that
\begin{align}\label{limit121}
  \left(\int_\mathbb{G}\left(\frac{1}{|\mathbb{B}(0,|x|)|}\int_{\mathbb{B}(0,|x|)}f^\alpha (y)\,dy\right)^\frac{q}{\alpha}u(x)\, dx\right)^\frac{1}{q} \le C_\alpha^{\frac{1}{\alpha}}\left(\int_\mathbb{G}f^p(x)\,v(x)\,dx\right)^\frac{1}{p},    
 \end{align}holds for all non-negative functions $f\in \mathbb{G}$ if 
 \begin{align*}
  \textbf{A}_{Q,\alpha}=\sup_{x\in\mathbb{G}}\left(\int_{\mathbb{G}\backslash\mathbb{B}(0,|x|)}u(y)\,|\mathbb{B}(0,|y|)|^{-\frac{q}{\alpha}}dy\right)^\frac{1}{q} \left(\int_{\mathbb{B}(0,|x|)}v^{\frac{\alpha}{\alpha-p}}(y)\,dy\right)^\frac{p-\alpha}{\alpha p}<\infty ,
 \end{align*}
 and the constant $C_\alpha$ satisfies the following estimate:
 \begin{align}\label{AQ1}
  C_\alpha^{\frac{1}{\alpha}}\le \textbf{A}_{Q,\alpha}\left(\frac{p}{p-\alpha}\right)^\frac{p-\alpha}{\alpha p}.\left(\frac{p}{\alpha}\right)^\frac{1}{q}.
 \end{align} 
  We note that
 \begin{align*}
  \textbf{A}_{Q,\alpha}\left(\frac{p}{p-\alpha}\right)^\frac{p-\alpha}{\alpha p}.\left(\frac{p}{\alpha}\right)^\frac{1}{q}=A_{Q,\alpha}\left(\frac{p}{p-\alpha}\right)^\frac{p-\alpha}{\alpha p} \times p^\frac{1}{q},  
\end{align*}
   with 
 \begin{align}\label{coff}
  A_{Q,\alpha}=\sup_{x\in \mathbb{G}}\left(\frac{1}{\alpha}\right)^\frac{1}{q}\times
  \left(\int_{\mathbb{G}\backslash\mathbb{B}(0,|x|)}u(y)|\mathbb{B}(0,|x|)|^\frac{-q}{\alpha}\,dy\right)^\frac{1}{q}\left(\int_{\mathbb{B}(0,|x|)}v^\frac{\alpha}{\alpha-p}(y)\,dy\right)^\frac{p-\alpha}{\alpha p}. 
   \end{align}

 Therefore, from \eqref{AQ1}, we have
 \begin{align}\label{AQ22}
  C_\alpha^{\frac{1}{\alpha}}\le A_{Q,\alpha}\left(\frac{p}{p-\alpha}\right)^\frac{p-\alpha}{\alpha p}\times p^\frac{1}{q}.  
 \end{align}
 
   Since
 \begin{align}\label{AQ2}
 \lim_{\alpha\rightarrow 0^+} \left(\frac{p}{p-\alpha}\right)^\frac{p-\alpha}{\alpha p}=e^\frac{1}{p}, 
 \end{align}
using \eqref{AQ2} in \eqref{AQ22}, we get
 \begin{align}\label{coeff6}
   C \le A_{Q} \,p^\frac{1}{q}\,e^\frac{1}{p} , 
 \end{align} where
 $$A_Q:=\lim_{\alpha\rightarrow 0^+} A_{Q,\alpha}\quad\quad \text{and} \quad\quad C:=\lim_{\alpha\rightarrow0^+}C_\alpha^{\frac{1}{\alpha}}.$$
 Recall that by  \eqref{measure} we have,
 \begin{align*}
 |\mathbb{B}(0,|y|)|:=\int_{\mathbb{B}(0,|y|)} dx=\int_\mathfrak{S}\left(\int_0^{|y|} r^{Q-1} dr\right) d\sigma=\frac{|y|^Q|\mathfrak{S}|}{Q}.     
 \end{align*}

 Now,  let us calculate the first integral from \eqref{coff}. We get
 \begin{align}\label{coeff1}
&\left(\int_{\mathbb{G}\backslash \mathbb{B}(0,|x|)}u(y)\,|\mathbb{B}(0,|y|)|^{\frac{-q}{\alpha}} dy \right)^\frac{1}{q}=\nonumber\left(\int_{\mathbb{G}\backslash\mathbb{B}(0,|x|)}u(y)\left(\frac{|y|^Q}{Q}|\mathfrak{S}|\right)^\frac{- q}{\alpha}dy\right)^\frac{1}{q}\\&=\nonumber\left(\left(\frac{|\mathfrak{S}|}{Q}\right)^\frac{-q}{\alpha}\int_{\mathbb{G}\backslash\mathbb{B}(0,|x|)}u(y)\frac{1}{|y|^\frac{Qq}{\alpha}} dy\right)^\frac{1}{q}\\&=\nonumber\left(\left(\frac{|\mathfrak{S}|}{Q}\right)^\frac{-q}{\alpha}\int_{\mathbb{G}\backslash\mathbb{B}(0,|x|)}u(y)\left(\frac{|x|}{|y|}\right)^\frac{Qq}{\alpha}|x|^\frac{-Qq}{\alpha} dy\right)^\frac{1}{q}\\&\nonumber =\left(\left(\frac{|\mathfrak{S}|}{Q}\right)^{\frac{-q}{\alpha}}\int_{\mathbb{G}\backslash \mathbb{B}(0,|x|)}u(y)\left(\frac{|x|}{|y|}\right)^{\frac{Qq}{\alpha}}|x|^Q\,|x|^{-Q}\,|x|^{\frac{-Qq}{\alpha}}\,dy\right)^{\frac{1}{q}}\\&=\left(\frac{|\mathfrak{S}|}{Q}\right)^\frac{-1}{\alpha}|x|^\frac{-Q}{\alpha}\,|x|^\frac{Q}{q}\left(\int_{\mathbb{G}\backslash\mathbb{B}(0,|x|)}u(y)\left(\frac{|x|}{|y|}\right)^\frac{Qq}{\alpha}|x|^{-Q} dy\right)^\frac{1}{q}.
\end{align}
Also, calculating the second integral of \eqref{coff}, we get that
\begin{align}\label{coeff2}
  &\left(\int_{\mathbb{B}(0,|x|)}v^\frac{\alpha}{\alpha-p}(y)\,dy\right)^{\frac{p-\alpha}{\alpha p}}=\nonumber\left(\frac{|\mathbb{B}(0,|x|)|}{|\mathbb{B}(0,|x|)|}\int_{\mathbb{B}(0,|x|)} v^{\frac{\alpha}{\alpha-p}}(y)\, dy\right)^{\frac{p-\alpha}{\alpha p}} \\&=\nonumber |\mathbb{B}(0,|x|)|^{\frac{p-\alpha}{\alpha p}}\left(\frac{1}{|\mathbb{B}(0,|x|)|}\int_{\mathbb{B}(0,|x|)} v^{\frac{\alpha}{\alpha-p}}(y)\, dy\right)^{\frac{p-\alpha}{\alpha p}}\\&=\nonumber\left(\frac{|x|^Q}{Q}|\mathfrak{S}|\right)^{\frac{p-\alpha}{\alpha p}}\left(\frac{1}{|\mathbb{B}(0,|x|)|}\int_{\mathbb{B}(0,|x|)} v^{\frac{\alpha}{\alpha-p}}(y)\, dy\right)^{\frac{p-\alpha}{\alpha p}}\\&=\left(\frac{|\mathfrak{S}|}{Q}\right)^{\frac{p-\alpha}{\alpha p}}|x|^{Q\left(\frac{p-\alpha}{\alpha p}\right)}\left(\frac{1}{|\mathbb{B}(0,|x|)|}\int_{\mathbb{B}(0,|x|)} v^{\frac{\alpha}{\alpha-p}}(y)\, dy\right)^{\frac{p-\alpha}{\alpha p}}.
\end{align}
Next, substituting the values from \eqref{coeff1} and \eqref{coeff2} in \eqref{coff}, we obtain
\begin{align*}
 &A_{Q,\alpha}=\sup_{x\in \mathbb{G}} \left(\frac{1}{\alpha}\right)^\frac{1}{q}\,\left(\frac{|\mathfrak{S}|}{Q}\right)^\frac{-1}{\alpha} \left(\frac{|\mathfrak{S}|}{Q}\right)^{\frac{p-\alpha}{\alpha p}}|x|^\frac{-Q}{\alpha}|x|^\frac{Q}{q}|x|^{Q\left(\frac{p-\alpha}{\alpha p}\right)}\times\\&\left(\int_{\mathbb{G}\backslash\mathbb{B}(0,|x|)}u(y)\left(\frac{|x|}{|y|}\right)^\frac{Qq}{\alpha}|x|^{-Q} dy\right)^\frac{1}{q}
\left(\frac{1}{|\mathbb{B}(0,|x|)|}\int_{\mathbb{B}(0,|x|)} v^{\frac{\alpha}{\alpha-p}}(y)\, dy\right)^{\frac{p-\alpha}{\alpha p}}.
\end{align*} Thus we find,
\begin{align}\label{coeff4}
   &A_{Q,\alpha}=\sup_{x\in\mathbb{G}}\left(\frac{1}{\alpha}\right)^\frac{1}{q}|x|^{\frac{Q}{q}-\frac{Q}{p}}\left(\frac{|\mathfrak{S}|}{Q}\right)^\frac{-1}{p}\times\\&\nonumber\left(\int_{\mathbb{G}\backslash\mathbb{B}(0,|x|)}u(y)\left(\frac{|x|}{|y|}\right)^\frac{Qq}{\alpha}|x|^{-Q} dy\right)^\frac{1}{q}
\left(\frac{1}{|\mathbb{B}(0,|x|)|}\int_{\mathbb{B}(0,|x|)} v^{\frac{\alpha}{\alpha-p}}(y)\, dy\right)^{\frac{p-\alpha}{\alpha p}}.
\end{align}
Performing the variable transformation $ y=|x|z$ in the first integral of \eqref{coeff4} and then using  the polar decomposition  by setting $z=r\omega,$ we have
\begin{align}\label{Variable}
   \nonumber I_\alpha(x)=\frac{1}{\alpha}\int_{\mathbb{G}\backslash \mathbb{B}(0,|x|)}u(y)\left(\frac{|x|}{|y|}\right)^\frac{Qq}{\alpha}|x|^{-Q} dy&=\frac{1}{\alpha}\int_{\mathbb{G}\backslash \mathbb{B}(0, 1)} u(|x|z)\,|z|^{\frac{-Qq}{\alpha}}dz
\\&=\frac{1}{\alpha}\int_{\mathfrak{S}}\int_1^\infty u\left(|x|r\omega\right)r^{Q-\frac{Qq}{\alpha}-1}dr\,d\omega.
\end{align}
We observe that \begin{align*}
  \chi _{(1,\infty)}(r)\, Q\left(\frac{q}{\alpha}-1\right)\,r^{Q-\frac{Qq}{\alpha}-1}\longrightarrow \delta_1(r)\quad\text{as}\quad \alpha\to{0}^+, 
\end{align*}
where $\delta_1(r)$ is the Dirac delta function at $r=1.$   Indeed, by choosing a test function $\phi \in C^\infty_c(\mathbb{R})$ we see that $\lim_{\alpha \rightarrow0^+}\langle \chi _{(1,\infty)}(r)\, Q\left(\frac{q}{\alpha}-1\right)\,r^{Q-\frac{Qq}{\alpha}-1}, \phi \rangle= \lim_{\alpha \rightarrow0^+} \int_{1}^\infty Q\left(\frac{q}{\alpha}-1\right)\,r^{Q-\frac{Qq}{\alpha}-1} \phi(r)\, dr= -\lim_{\alpha \rightarrow0^+} ( r^{Q-\frac{Qq}{\alpha}} \phi(r)|_{r=1}^{\infty})+ \lim_{\alpha \rightarrow0^+} \int_{1}^\infty r^{Q-\frac{Qq}{\alpha}} \phi'(r) dr = \phi(1)=\langle \delta_1, \phi \rangle.$ This implies that 
$$\lim_{\alpha \rightarrow0^+}  \chi _{(1,\infty)}(r)\, Q\left(\frac{q}{\alpha}-1\right)\,r^{Q-\frac{Qq}{\alpha}-1}= \delta_1(r).$$

Indeed, for  $r \in (1, \infty)$ and for $\alpha\rightarrow 0^+,$ we have $\lim_{\alpha \rightarrow0^+} r^{Q-\frac{Qq}{\alpha}}=0,$
 which shows that
 $$\lim_{\alpha \rightarrow0^+}  \chi _{(1,\infty)}(r)\, \frac{1}{\alpha}\,r^{Q-\frac{Qq}{\alpha}-1}= \frac{\delta_1(r)}{Qq}.$$
Thus, from  \eqref{Variable}, we have
\begin{align}\label{coeff5}
I_\alpha(x)\rightarrow  \frac{1}{Qq} \int_\mathfrak{S} u\left(|x|\omega\right)\,d\omega=\frac{1}{Qq}\,|\mathfrak{S}|\,u_1(x)\quad \text{as}\quad \alpha\to{0}^+.
\end{align} A simple calculation gives, as $\beta \to{0}^+,$ that
\begin{align}\label{limit}
    \left(\frac{1}{|\mathbb{B}(0,|x|)|}\int_{\mathbb{B}(0,|x|)}f^\beta(y)\,dy\right)^\frac{1}{\beta} \rightarrow \exp\left(\frac{1}{|\mathbb{B}(0,|x|)|}\int_{\mathbb{B}(0,|x|)}\log\,f(y)\, dy\right).
\end{align}
Next, using \eqref{limit} in the second integral of \eqref{coeff4}, we have
\begin{align}\label{limit11}
 \left(\frac{1}{|\mathbb{B}(0,|x|)|}\int_{\mathbb{B}(0,|x|)} v^{\frac{\alpha}{\alpha-p}}(y)\, dy\right)^{\frac{p-\alpha}{\alpha p}}\longrightarrow\exp\left(\frac{1}{|\mathbb{B}(0,|x|)|}\int_{\mathbb{B}(0,|x|)}\log\frac{1}{v(y)}\,dy\right)^\frac{1}{p}\quad\text{as}\quad\alpha\to{0}^+.
\end{align}
Substituting \eqref{coeff5} and \eqref{limit11} in \eqref{coeff4} as $\alpha \to{0}^+$, we have
\begin{align}\label{AQ}
\nonumber &A_Q=\lim_{\alpha\to{0}^+} A_{Q,\alpha}\\&\nonumber=\nonumber\sup_{x\in \mathbb{G}}|x|^{\frac{Q}{q}-\frac{Q}{p}}\left(\frac{|\mathfrak{S}|}{Q}\right)^{\frac{-1}{p}} \left(\frac{1}{Qq}|\mathfrak{S}|u_1(x)\right)^\frac{1}{q}\left(\exp\left(\frac{1}{|\mathbb{B}(0,|x|)|}\int_{\mathbb{B}(0,|x|)}\log\frac{1}{v(y)}dy\right)\right)^\frac{1}{p}\\&=\nonumber\sup_{x\in \mathbb{G}}q^\frac{-1}{q}|x|^{\frac{Q}{q}-\frac{Q}{p}}\left(\frac{|\mathfrak{S}|}{Q}\right)^{\frac{1}{q}-\frac{1}{p}} u^\frac{1}{q}_1(x)\left(\exp\left(\frac{1}{|\mathbb{B}(0,|x|)|}\int_{\mathbb{B}(0,|x|)}\log\frac{1}{v(y)}dy\right)\right)^\frac{1}{p}\\&=\nonumber\sup_{x\in \mathbb{G}} q^{-\frac{1}{q}}\left(|x|^Q\frac{|\mathfrak{S}|}{Q}\right)^{\frac{1}{q}-\frac{1}{p}}\,u_1^\frac{1}{q}(x)\left(\exp\left(\frac{1}{|\mathbb{B}(0,|x|)|}\int_{\mathbb{B}(0,|x|)}\log\frac{1}{v(y)}dy\right)\right)^\frac{1}{p}\\&\quad\quad\quad\quad\quad=q^{-\frac{1}{q}}D_Q,
\end{align}
which is finite by the hypothesis of the theorem.

Thus, putting the value of $A_Q$ from \eqref{AQ} in \eqref{coeff6}, we get
\begin{align}
 0<C\le\left(\frac{p}{q}\right)^\frac{1}{q} \,\exp\bigg(\frac{1}{p}\bigg)\,D_Q,
\end{align} which is same as \eqref{DQ}.\\
Finally,  using \eqref{limit} in \eqref{limit121}, we obtain
\begin{align}
 \left(\int_{\mathbb{G}} \left[ \exp \left( \frac{1}{|\mathbb{B}(0, |x|)|} \int_{\mathbb{B}(0, |x|)} \log f(y)\, dy \right) \right]^q u(x)\, dx \right) ^\frac{1}{q}\leq C \left(\int_{\mathbb{G}} f^p(x)\, v(x)\,dx \right)^{\frac{1}{p}},  
\end{align}
 completing the proof of the theorem.
\end{proof}
In the next Theorem \ref{THM2}, we will prove the following stronger inequality for the special case $u(x)=|\mathbb{B}(0,|x|)|^a$ and $v(x)=|\mathbb{B}(0,|x|)|^b$:
\begin{theorem}\label{THM2} Let $\mathbb{G}$ be a homogeneous group with the homogeneous dimension  $Q$ equipped with a  quasi norm $|\cdot|.$
Let $0<p\le q<\infty$ and $a,b\in \mathbb{R}.$ Then for any $\epsilon>0$ and  for any arbitrary positive function $f$ on the homogeneous Lie group $\mathbb{G},$ the following inequality
\begin{align}\label{Thm2}
   \nonumber \Bigg(\int_\mathbb{G}\Bigg[\exp\Bigg(\epsilon|\mathbb{B}(0,|x|)|^{-\epsilon}\int_{\mathbb{B}(0,|x|)}&|\mathbb{B}(0,|y|)|^{\epsilon-1}\log\,f(y) dy\Bigg)\Bigg]^q |\mathbb{B}(0,|x|)|^a dx\Bigg)^\frac{1}{q}\\&\le C\Bigg(\int_\mathbb{G} f^p(x)|\mathbb{B}(0,|x|)|^b\, dx\Bigg)^\frac{1}{p}
\end{align}holds for a positive finite constant $C$ if and only if \begin{align}\label{coeff7}
    p\left(a+1\right)-q\left(b+1\right)=0.
\end{align}
Moreover, the best constant $C$ in \eqref{Thm2} satisfies 
\begin{align}
  \left(\frac{p}{q}\right)^\frac{1}{q} \epsilon^{\frac{1}{p}-\frac{1}{q}}\, \exp\bigg({\frac{b+1}{\epsilon p}-\frac{1}{p}}\bigg) \le C \le\left(\frac{p}{q}\right)^\frac{1}{q}\epsilon^{\frac{1}{p}-\frac{1}{q}}\, \exp\bigg({\frac{b+1}{\epsilon p}}\bigg).
\end{align}
\end{theorem}
\begin{proof} Assume that equality  \eqref{coeff7} holds. Let $f$  be an arbitrary positive function on the homogeneous Lie group $\mathbb{G}$ equipped with a quasi norm $|\cdot|$. To prove \eqref{Thm2} we first obtain its equivalent form using polar decomposition on $\mathbb{G}$ and then we use 
Theorem \ref{Th2} to establish it. Indeed, using the polar decomposition on $\mathbb{G}$, by setting $x=r\sigma$ and $y=t \tau$  in \eqref{Thm2} on $\mathbb{G},$  we get an equivalent inequality of  \eqref{Thm2}:
\begin{align}\label{Vt}
  \nonumber\Bigg(\int_\mathfrak{S}\int_0^\infty \left(\frac{|\mathfrak{S}|}{Q}\right)^a & r^{Qa+Q-1}\Bigg[\exp\bigg(\epsilon r^{-Q\epsilon}\left(\frac{|\mathfrak{S}|}{Q}\right)^{-\epsilon} \times \\&  \int_\mathfrak{S}\int_{0}^r\left(\frac{|\mathfrak{S}|}{Q}\right)^{\epsilon-1} t^{Q\epsilon-1}\log f(t\tau) dt\,d\tau \bigg)\Bigg]^q dr\,d\sigma  \Bigg)^\frac{1}{q} \nonumber\\&\quad\quad\quad \le C\left(\int_\mathfrak{S}\int_0^\infty r^{Qb+Q-1}\left(\frac{|\mathfrak{S}|}{Q}\right)^b f^p(r \sigma)\,dr\,d\sigma \right)^\frac{1}{p}.
\end{align}
Next, we do variable transformation $r=r_1^{\frac{1}{\epsilon}}$ and $t=t_1^\frac{1}{\epsilon}$ in \eqref{Vt} to obtain 
\begin{align}\label{in11}
 \nonumber &\Bigg(\int_\mathfrak{S}\int_0^\infty\left(\frac{|\mathfrak{S}|}{Q}\right)^a r_1^{Q\left({\frac{a+1}{\epsilon}-1}\right)}r_1^{Q-1}\bigg[\exp\Bigg(\frac{Qq}{|\mathfrak{S}|r_1^Q}\quad  \times \\&\quad\quad \int_\mathfrak{S}\int_0^{r_1}t_1^{Q-1}\,\times \log f\left(t_1^\frac{1}{\epsilon}\tau \right)\,dt_1\,d\tau\Bigg)\bigg]^q\frac{1}{\epsilon}\,dr_1 d\sigma\Bigg)^\frac{1}{q} \nonumber \\&\quad\quad\quad\quad\quad\le  C\left(\int_\mathfrak{S}\int_0^\infty\left(\frac{|\mathfrak{S}|}{Q}\right)^br_1^{Q\left({\frac{b+1}{\epsilon}-1}\right)}r_1^{Q-1}\,f^p(r_1^\frac{1}{\epsilon}\sigma)\frac{1}{\epsilon}\,dr_1 \,d\sigma \right)^\frac{1}{p}.
\end{align}
Recalling the volume of $|\mathbb{B}(0, |z|)|$ in $\mathbb{G}$ from \eqref{measure}, that is,
$  |\mathbb{B}(0, |z|)|=\frac{|z|^Q|\mathfrak{S}|}{Q} $
and 
using this in \eqref{in11}, we have
\begin{align} \label{anj}
&\nonumber\Bigg(\int_\mathfrak{S}\int_0^\infty |\mathbb{B}(0, r_1)|^{\left(\frac{a+1}{\epsilon}-1\right)} \,\,\times \\&\quad\quad\bigg[\exp\left(\frac{q}{|\mathbb{B}(0,r_1)|}\int_{\mathfrak{S}}\int_0^{r_1}\log\,F(t_1 \tau) t_1^{Q-1}\,dt_1 d\tau\right)\bigg]^qr_1^{Q-1}dr_1 d\sigma \Bigg)^{\frac{1}{q}} \nonumber \\&\le
C\epsilon^{\frac{1}{q}-\frac{1}{p}} \left(\frac{|\mathfrak{S}|}{Q}\right)^{\left(\frac{b+1}{p}-\frac{a+1}{q}\right)\left(1-\frac{1}{\epsilon}\right)}\left(\int_\mathfrak{S}\int_0^\infty|\mathbb{B}(0,r_1)|^{\left(\frac{b+1}{\epsilon}-1\right)}F^p(r_1\sigma) r_1^{Q-1}\,dr_1 d\sigma\right)^\frac{1}{p},
\end{align}
where we have written $F(r\sigma)=f(r^\frac{1}{\epsilon}\sigma).$

Again using the polar decomposition in  $\mathbb{G}$ with $t_1\tau=z$ and $r_1\sigma=w$, the inequality \eqref{anj} yields that
\begin{align}\label{in12}
   &\nonumber\left(\int_\mathbb{G} |\mathbb{B}(0,|w|)|^{\left(\frac{a+1}{\epsilon}-1\right)} \left[\exp\bigg(\frac{q}{|\mathbb{B}(0,|w|)|}\int_{\mathbb{B}(0,|w|)}\log\,F(z)\,dz\bigg)\right]^q dw\right)^{\frac{1}{q}}\\&\quad \quad\le
C\epsilon^{\frac{1}{q}-\frac{1}{p}}\left(\frac{|\mathfrak{S}|}{Q}\right)^{\left(\frac{b+1}{p}-\frac{a+1}{q}\right)\left(1-\frac{1}{\epsilon}\right)}\left(\int_\mathbb{G}|\mathbb{B}(0,|w|)|^{\left(\frac{b+1}{\epsilon}-1\right)}F^p(w)\,dw\right)^\frac{1}{p}. 
\end{align}
Recalling  the assumption  from \eqref{coeff7} that, $ \frac{b+1}{p}-\frac{a+1}{q}=0,$ we get

\begin{align}\label{inq1}
  &\nonumber\left(\int_\mathbb{G} |\mathbb{B}(0,|w|)|^{\left(\frac{a+1}{\epsilon}-1\right)} \left[\exp\bigg(\frac{q}{|\mathbb{B}(0,|w|)|}\int_{\mathbb{B}(0,|w|)}\log\,F(z)\,dz\bigg)\right]^q dw\right)^{\frac{1}{q}}\\&\quad \quad\le
C\epsilon^{\frac{1}{q}-\frac{1}{p}}\left(\int_\mathbb{G}|\mathbb{B}(0,|w|)|^{\left(\frac{b+1}{\epsilon}-1\right)}F^p(w)\,dw\right)^\frac{1}{p}.
\end{align}

Now, we note that the above inequality \eqref{inq1} is equivalent to the inequality \eqref{Thm2}. Therefore, to prove \eqref{Thm2}, it is enough to show that $D_Q$ (ref. \eqref{coeff33}) is finite.
Thus, we apply Theorem \ref{Th2} with the corresponding weights $u(w)=|\mathbb{B}(0,|w|)|^{\left(\frac{a+1}{\epsilon}-1\right)}$ and $v(w)=|\mathbb{B}(0,|w|)|^{\left(\frac{b+1}{\epsilon}-1\right)}$
on $\mathbb{G}.$ For this case,
\begin{align}\label{COEFF}
    D_Q=\sup_{w\in \mathbb{G}}|\mathbb{B}(0,|w|)|^{\frac{1}{q}-\frac{1}{p}}\,u^{\frac{1}{q}}_1(w)\left[\exp\left(\frac{1}{|\mathbb{B}(0,|w|)|}\int_{\mathbb{B}(0,|w|)}\log\frac{1}{v(s)}ds\right)\right]^{\frac{1}{p}},
\end{align} 
where $u_1(w)=\frac{1}{|\mathfrak{S}|}\int_{\mathfrak{S}} u(|w|\sigma)\,d\sigma.$
Let us now calculate the value of $u_1.$ In fact, we have
\begin{align}\label{COEFF2}
\nonumber u^{\frac{1}{q}}_1(w)&=\left(\frac{1}{|\mathfrak{S}|}\int_{\mathfrak{S}} u(|w|\sigma)\,d\sigma\right)^\frac{1}{q}\\& \nonumber=\left(\frac{1}{|\mathfrak{S}|}\right)^\frac{1}{q}\left(\int_\mathfrak{S}|\mathbb{B}(0,|w|)|^{\left(\frac{a+1}{\epsilon}-1\right)}d\sigma\right)^\frac{1}{q}\\&=\left(\frac{1}{|\mathfrak{S}|}\right)^\frac{1}{q}|\mathbb{B}(0,|w|)|^{\frac{1}{q}\left(\frac{a+1}{\epsilon}-1\right)}\left(\int_\mathfrak{S}d\sigma\right)^\frac{1}{q}=|\mathbb{B}(0,|w|)|^{\frac{1}{q}\left(\frac{a+1}{\epsilon}-1\right)}. 
\end{align} 
Now, we calculate the value of the integral in right hand side of \eqref{COEFF} by using the polar decomposition $s=(t,w)$ with $|s|=t$ as follows:
\begin{align}\label{PD}
 \nonumber\int_{\mathbb{B}(0,|w|)}\log\frac{1}{v(s)}\,ds &=\int_{\mathbb{B}(0,|w|)} \log|\mathbb{B}(0,|s|)|^{-\left(\frac{b+1}{\epsilon}-1\right)} ds\nonumber \\&\nonumber=\int_\mathfrak{S} \int_0^{|w|} \left(1-\frac{b+1}{\epsilon}\right)t^{Q-1}\log\left(\frac{|\mathfrak{S}|}{Q} t^Q\right) dt\,dw \\&=\left(1-\frac{b+1}{\epsilon}\right)\int_0^{|w|}|\mathfrak{S}|\,t^{Q-1}\log\left(\frac{|\mathfrak{S}|}{Q} t^Q\right)dt.
\end{align}
Observe that,
\begin{align}\label{eq3}
 \nonumber&\int_0^{\frac{|\mathfrak{S}|}{Q}|w|^Q}\log U dU=\log U\times U\bigg|_0^{\frac{|\mathfrak{S}|}{Q}|w|^Q}-\int_0^{\frac{|\mathfrak{S}|}{Q}|w|^Q}\frac{d}{dU}(\log U)\times U \\&\quad \quad \quad \quad = \left(\log \frac{|\mathfrak{S}|}{Q}|w|^Q\right)\left(\frac{|\mathfrak{S}|}{Q}|w|^Q\right)-\left(\frac{|\mathfrak{S}|}{Q}|w|^Q\right).
\end{align}
Using \eqref{eq3} in \eqref{PD} along with the change of variable $\frac{|\mathfrak{S}|}{Q} t^Q=U$ and using $  |\mathbb{B}(0, |w|)|=\frac{|w|^Q|\mathfrak{S}|}{Q},$ we get
\begin{align}\label{COEFF1}
  \nonumber\int_{\mathbb{B}(0,|w|)}\log\frac{1}{v(s)}\,ds&=\left(1-\frac{b+1}{\epsilon}\right)\int_0^{\frac{|\mathfrak{S}|}{Q}|w|^Q}\log U dU
  \\&\nonumber=\left(1-\frac{b+1}{\epsilon}\right)\left[\left(\log \frac{|\mathfrak{S}|}{Q}|w|^Q\right)\left(\frac{|\mathfrak{S}|}{Q}|w|^Q\right)-\left(\frac{|\mathfrak{S}|}{Q}|w|^Q\right)\right]\\&\nonumber=\left(1-\frac{b+1}{\epsilon}\right)\left(\frac{|\mathfrak{S}|}{Q} |w|^Q\right) \left[\log \left(\frac{|\mathfrak{S}|}{Q} |w|^Q\right) -1\right]\nonumber\\& =|\mathbb{B}(0,|w|)|\left(\log|\mathbb{B}(0,|w|)|^{\left(1-\frac{b+1}{\epsilon}\right)}+\frac{b+1}{\epsilon}-1\right).
\end{align}
Substituting the values from \eqref{COEFF2} and \eqref{COEFF1} in \eqref{COEFF}, we obtain
\begin{align}
 &\nonumber D_Q=\sup_{w\in \mathbb{G}}|\mathbb{B}(0,|w|)|^{\frac{1}{q}-\frac{1}{p}}\,|\mathbb{B}(0,|w|)|^{\frac{1}{q}\left(\frac{a+1}{\epsilon}-1\right)}\\&\nonumber\left[\exp\left(\frac{1}{|\mathbb{B}(0,|w|)|}\,|\mathbb{B}(0,|w|)|\left(\log|\mathbb{B}(0,|w|)|^{\left(1-\frac{b+1}{\epsilon}\right)}+\frac{b+1}{\epsilon}-1\right) \right)\right]^\frac{1}{p}\\&\nonumber=\sup_{w\in\mathbb{G}}|\mathbb{B}(0,|w|)|^{\left(\frac{a+1}{q\epsilon}-\frac{1}{p}\right)}\left[\exp\left(\log|\mathbb{B}(0,|w|)|^{1-\frac{b+1}{\epsilon}}+\frac{b+1}{\epsilon}-1\right)\right]^\frac{1}{p}
 \\&=e^{\frac{1}{p}\left(\frac{b+1}{\epsilon}-1\right)}\sup_{w\in\mathbb{G}}|\mathbb{B}(0,|w|)|^{\left(\frac{a+1}{q\epsilon}-\frac{1}{p}\right)}|\mathbb{B}(0,|w|)|^{\frac{1}{p}-\frac{b+1}{\epsilon p}},
 \end{align}
 which implies that
 \begin{align}\label{coeff0}
  D_Q=\exp\bigg({\frac{1}{p}\left(\frac{b+1}{\epsilon}-1\right)}\bigg)\sup_{w\in\mathbb{G}}|\mathbb{B}(0,|w|)|^{\frac{1}{\epsilon}\left(\frac{a+1}{q}-\frac{b+1}{p}\right)}.
\end{align}
Thus, using the assumption \eqref{coeff7} in \eqref{coeff0}, we have
\begin{align}\label{DQ111}
   D_Q=\exp\bigg({\frac{1}{p}\left(\frac{b+1}{\epsilon}-1\right)}\bigg),
\end{align} which is finite.
Therefore, by Theorem \ref{Th2} inequality \eqref{Thm2} holds  for each positive function $f$ defined on $\mathbb{G}.$

Moreover, by  \eqref{DQ} and using the explicit value of $D_Q$ from \eqref{DQ111}, we get
\begin{align}
  \epsilon^{\frac{1}{q}-\frac{1}{p}} C\le \left(\frac{p}{q}\right)^\frac{1}{q}\,\exp\bigg({\frac{b+1}{\epsilon p}}\bigg)\,\,\text{implies that}\,\, C\le \left(\frac{p}{q}\right)^\frac{1}{q}\epsilon^{\frac{1}{p}-\frac{1}{q}}\,\exp\bigg({\frac{b+1}{\epsilon p}}\bigg).
\end{align}

Conversely, we assume that \eqref{Thm2} holds for all positive functions $f$ on  $\mathbb{G}$. Again we will use the equivalent form \eqref{in12} of   \eqref{Thm2}. It is clear from \eqref{in12} that the following inequality is true for any ball $\mathbb{B}(0, |x|)$ of radius  $x \in \mathbb{G}:$
\begin{align}\label{in122}
   &\nonumber\left(\int_\mathbb{\mathbb{B}(0, |x|)} |\mathbb{B}(0,|w|)|^{\left(\frac{a+1}{\epsilon}-1\right)} \left[\exp\frac{1}{|\mathbb{B}(0,|w|)|}\int_{\mathbb{B}(0,|w|)}\log\,F(z)\,dz\right]^q dw\right)^{\frac{1}{q}}\\&\quad \quad\le
C\epsilon^{\frac{1}{q}-\frac{1}{p}}\left(\frac{|\mathfrak{S}|}{Q}\right)^{\left(\frac{b+1}{p}-\frac{a+1}{q}\right)\left(1-\frac{1}{\epsilon}\right)}\left(\int_\mathbb{G}|\mathbb{B}(0,|w|)|^{\left(\frac{b+1}{\epsilon}-1\right)}F^p(w)\,dw\right)^\frac{1}{p}. 
\end{align}

So, we will test the inequality \eqref{in122} (and therefore, equivalently, inequality \eqref{Thm2}) with the test function, for $x \in \mathbb{G}$ as chosen above,
\begin{align*}
 F_x(z) =|\mathbb{B}(0,|z|)|^{\frac{1}{p}\left(1-\frac{b+1}{\epsilon}\right)} \,\chi_{\mathbb{B} (0,|x|)},\quad\quad z\in \mathbb{G},
\end{align*}
where $\chi_{\mathbb{B} (0,|x|)}$ is the characteristic function of $\mathbb{B} (0,|x|)$ in $\mathbb{G}.$ Indeed, we  obtain by noting $|z|\leq |w|\leq |x|$ that
\begin{align}\label{COEFF6}
 &\nonumber\left(\int_{\mathbb{B}(0,|x|)} |\mathbb{B}(0,|w|)|^{\left(\frac{a+1}{\epsilon}-1\right)}\left[\exp\frac{1}{|\mathbb{B}(0,|w|)|}\int_{\mathbb{B}(0,|w|)}\log|\mathbb{B}(0,|z|)|^{\frac{1}{p}\left(1-\frac{b+1}{\epsilon}\right)}\,dz\right]^q dw\right)^{\frac{1}{q}}\\&\quad\quad\quad\quad\quad\quad\le
C\epsilon^{\frac{1}{q}-\frac{1}{p}}\left(\frac{|\mathfrak{S}|}{Q}\right)^{\left(\frac{b+1}{p}-\frac{a+1}{q}\right)\left(1-\frac{1}{\epsilon}\right)}|\mathbb{B}(0,|x|)|^\frac{1}{p}  =\tilde{C}\,|\mathbb{B}(0,|x|)|^\frac{1}{p},
\end{align}

where\begin{align}\label{VOC}
 \tilde{C}= C\epsilon^{\frac{1}{q}-\frac{1}{p}}\left(\frac{|\mathfrak{S}|}{Q}\right)^{\left(\frac{b+1}{p}-\frac{a+1}{q}\right)\left(1-\frac{1}{\epsilon}\right)}.
\end{align}
Again, using the polar decomposition on $\mathbb{G}$ and doing calculations similar to  \eqref{PD}, we get
\begin{align}\label{COEFF5}
 \int_{\mathbb{B}(0,|w|)}\log|\mathbb{B}(0,|z|)|^{\frac{1}{p}\left(1-\frac{b+1}{\epsilon}\right)}dz=|\mathbb{B}(0,|w|)|\left(\log|\mathbb{B}(0,|w|)|^{\frac{1}{p}\left(1-\frac{b+1}{\epsilon}\right)}+\frac{1}{p}\left(\frac{b+1}{\epsilon}-1\right)\right).    
\end{align}
Using \eqref{COEFF5} in \eqref{COEFF6} we obtain
\begin{align*}
  &\Bigg(\int_{\mathbb{B}(0,|x|)}|\mathbb{B}(0,|w|)|^{\frac{a+1}{\epsilon}-1}\Bigg(\exp \frac{1}{|\mathbb{B}(0,|w|)|}\Bigg\{ |\mathbb{B}(0,|w|)|\times\\&\nonumber\quad\quad\left(\log|\mathbb{B}(0,|w|)|^{\frac{1}{p}\left(1-\frac{b+1}{\epsilon}\right)}+\frac{1}{p}\left(\frac{b+1}{\epsilon}-1\right)\right)\Bigg\}\Bigg)^q dw   \Bigg)^\frac{1}{q}\le \tilde{C}|\mathbb{B}(0,|x|)|^\frac{1}{p}.
\end{align*}
This implies that

\begin{align}\label{F1}
  \nonumber&\left(\int_{\mathbb{B}(0,|x|)}|\mathbb{B}(0,|w|)|^{\frac{a+1}{\epsilon}-1}\left(|\mathbb{B}(0,|w|)|^{\frac{1}{p}\left(1-\frac{b+1}{\epsilon}\right)}\,\exp\left(\frac{1}{p}\left(\frac{b+1}{\epsilon}-1\right)\right)\right)^q dw\right)^\frac{1}{q}\\&\quad\quad\quad\quad\quad\quad\quad\quad\le \tilde{C}|\mathbb{B}(0,|x|)|^\frac{1}{p}.   
\end{align}
Therefore, from \eqref{F1}, we can rewrite \eqref{COEFF6} in the following form 
\begin{align*}
 &\exp\left\{\frac{b+1}{\epsilon p}-\frac{1}{p}\right\}\left(\int_{\mathbb{B}(0,|x|)}|\mathbb{B}(0,|w|)|^{\frac{a+1}{\epsilon}-1+\frac{q}{p}\left(1-\frac{b+1}{\epsilon}\right)}\, dw\right)^\frac{1}{q}\le \tilde{C}\, |\mathbb{B}(0,|x|)|^\frac{1}{p},
\end{align*}
which further can be rewritten  as
\begin{align} \label{eq339}
  \left(\int_{\mathbb{B}(0,|x|)}|\mathbb{B}(0,|w|)|^{\left(\frac{q}{p}\left(\frac{a+1}{\epsilon}\frac{p}{q}-\frac{b+1}{\epsilon}+1\right)-1\right)} dw\right)^\frac{1}{q}\le \tilde{C}\,\exp \left\{\frac{1}{p}-\frac{b+1}{\epsilon p} \right\} |\mathbb{B}(0,|x|)|^\frac{1}{p}. 
\end{align}
Therefore, the inequality \eqref{eq339} implies that
\begin{align}\label{in3}
    \left(\int_{\mathbb{B}(0,|x|)}|\mathbb{B}(0,|w|)|^{\frac{q}{p_0}-1} dw\right)^\frac{1}{q}\le \tilde{C}\,\exp\left\{\frac{1}{p}-\frac{b+1}{\epsilon p}\right\} |\mathbb{B}(0,|x|)|^\frac{1}{p},
\end{align}
where 
 \begin{align}\label{COEFF11}
    p_0=\frac{p}{\frac{a+1}{\epsilon}\frac{p}{q}-\frac{b+1}{\epsilon}+1}.
\end{align}
Now, again using the polar decomposition on $\mathbb{G},$  from \eqref{in3} we obtain 
\begin{align*}
 &\nonumber\left(\int_{\mathbb{B}(0,|x|)}|\mathbb{B}(0,|w|)|^{\frac{q}{p_0}-1}dw\right)^\frac{1}{q}=\left(\left(\frac{|\mathfrak{S}|}{Q}\right)^{\frac{q}{p_0}-1}\int_{\mathbb{B}(0,|x|)} |w|^{Q(\frac{q}{p_0}-1)}dw\right)^\frac{1}{q}\\&\nonumber\quad\quad\quad\quad=\left(\left(\frac{|\mathfrak{S}|}{Q}\right)^{\frac{q}{p_0}-1}\int_{\mathfrak{S}}\int_0^{|x|} r^{Q-1}\,r^{Q(\frac{q}{p_0}-1)}dr \,d\rho\right)^{\frac{1}{q}}\\&\nonumber\quad\quad\quad\quad=\left(\left(\frac{|\mathfrak{S}|}{Q}\right)^{\frac{q}{p_0}-1}\int_0^{|x|} |\mathfrak{S}|\,r^{Q\frac{q}{p_0}-1}dr \right)^{\frac{1}{q}}=\left(\left(\frac{|\mathfrak{S}|}{Q}\right)^{\frac{q}{p_0}-1} |\mathfrak{S}|\,|x|^{Q\frac{q}{p_0}}\left(\frac{p_0}{Qq}\right) \right)^{\frac{1}{q}}\\&\quad\quad\quad\quad\quad=\left(\frac{p_0}{q}\right)^\frac{1}{q}\left(\frac{|x|^Q\,|\mathfrak{S}|}{Q}\right)^\frac{1}{p_0}=|\mathbb{B}(0,|x|)|^\frac{1}{p_0}\left(\frac{p_0}{q}\right)^\frac{1}{q}.
\end{align*}
Thus we have
\begin{align}\label{f}
\left(\int_{\mathbb{B}(0,|x|)}|\mathbb{B}(0,|w|)|^{\frac{q}{p_0}-1}dw\right)^\frac{1}{q}=|\mathbb{B}(0,|x|)|^\frac{1}{p_0}\left(\frac{p_0}{q}\right)^\frac{1}{q}.
\end{align}
Therefore, using \eqref{f} in \eqref{in3}, we deduce that
\begin{align*}
    |\mathbb{B}(0,|x|)|^{\frac{1}{p_0}}\left(\frac{p_0}{q}\right)^\frac{1}{q}\le\tilde{C}\,\exp\bigg({\frac{1}{p}-\frac{b+1}{\epsilon p}}\bigg)\,|\mathbb{B}(0,|x|)|^\frac{1}{p},
\end{align*} 
which implies that
\begin{align}\label{eq123} |\mathbb{B}(0,|x|)|^{\frac{1}{p_0}-\frac{1}{p}}\le \tilde{C}\,\exp\bigg({\frac{1}{p}-\frac{b+1}{\epsilon p}}\bigg)\left(\frac{q}{p_0}\right)^\frac{1}{q}.
\end{align}
Since \eqref{eq123} holds for every $x \in \mathbb{G},$ this implies that  $p_0=p,$ that is, from \eqref{COEFF11}, 
 $$\frac{a+1}{\epsilon}\frac{p}{q}-\frac{b+1}{\epsilon}+1=1.$$
 Hence \begin{align}\label{final} 
     \frac{a+1}{q}-\frac{b+1}{p}=0,
     \end{align} which is same as \eqref{coeff7}. 
     
     Again substituting the value of $\tilde{C}$
 from \eqref{VOC} in \eqref{eq123}, we estimate that
 \begin{align*}
  |\mathbb{B}(0,|x|)|^{\frac{1}{p_0}-\frac{1}{p}}\le \exp\bigg({\frac{1}{p}-\frac{b+1}{\epsilon p}}\bigg)\left(\frac{q}{p_0}\right)^\frac{1}{q}  C\epsilon^{\frac{1}{q}-\frac{1}{p}}\left(\frac{|\mathfrak{S}|}{Q}\right)^{\left(\frac{b+1}{p}-\frac{a+1}{q}\right)\left(1-\frac{1}{\epsilon}\right)} 
 \end{align*}
which with the fact that $p=p_0$ and  \eqref{final} gives that
 \begin{align}
  C\ge \left(\frac{p}{q}\right)^\frac{1}{q}\,\exp\bigg({\frac{b+1}{\epsilon p}-\frac{1}{p}}\bigg) \epsilon^{\frac{1}{p}-\frac{1}{q}}.
 \end{align}
 This completes the proof of the theorem. \end{proof}
Next, we also state the following result which is conjugate of Theorem
\ref{Th2}:
\begin{theorem} \label{congthm}
Let  $u$ and $v$ be weight functions on a homogeneous Lie group $\mathbb{G}$ of homogeneous dimension $Q,$ equipped with an arbitrary quasi norm $|\cdot|,$ and let  $0<p\leq q<\infty$ and $\epsilon>0.$  Then, there exist a positive constant $C$ such that, for all positive functions $f$ on $\mathbb{G}$, the following inequality holds 
\begin{align} \label{main2}
    &\nonumber \left(\int_{\mathbb{G}} \left[ \exp \left( \epsilon\left|\mathbb{B}(0, |x|)\right|^\epsilon \int_{\mathbb{G}\backslash\mathbb{B}(0, |x|)}\left|\mathbb{B}(0,|x|)\right|^{-\epsilon-1} \log f(y)\, dy \right) \right]^q u(x)\, dx \right)^\frac{1}{q}  \\&\quad\quad\quad\quad\quad \le C \left(\int_{\mathbb{G}} f^p(x)\, v(x)\,dx \right)^{\frac{1}{p}},
\end{align}
provided that 
\begin{equation}
   \Tilde{D}_Q:= \sup_{x \in \mathbb{G}} |\mathbb{B}(0, |x|)|^{\frac{1}{q}-\frac{1}{p}}\, \tilde{u}^{\frac{1}{q}}(x) \left[ \exp \left( \frac{1}{|\mathbb{B}(0, |x|)|} \int_{\mathbb{B}(0, |x|)} \log \frac{1}{\tilde{v}(y)}\,dy \right)\right]^{\frac{1}{p}}<\infty,
\end{equation} where $\tilde{u}$ and $\tilde{v}$ are the spherical average of $u$ and $v$ respectively, given by, 
$$\tilde{u}(s)=u(s^{-\frac{1}{\epsilon}})\frac{1}{\epsilon}|s|^{-Q\left(1+\frac{1}{\epsilon}\right)}\quad\quad \tilde{v}(s)=v(s^{-\frac{1}{\epsilon}})\frac{1}{\epsilon}|s|^{-Q\left(1+\frac{1}{\epsilon}\right)}.$$

Moreover, the optimal constant $C$ in \eqref{main2} can estimated as follows:
\begin{equation}\label{DQ1}
    0<C \leq \left(\frac{p}{q} \right)^{\frac{1}{q}} \exp\bigg({\frac{1}{p}}\bigg) \tilde{D}_Q.
\end{equation}
\end{theorem}
\begin{proof}
 Let $f$ be a positive function on the homogeneous group $\mathbb{G}$ with an arbitrary quasi norm $|\cdot|$. Now, using the polar decomposition of \eqref{main2} on the homogeneous group $\mathbb{G}$  similar to \eqref{Vt} in Theorem \ref{THM2} and again, making some variable transformations similar to \eqref{in11} and using \eqref{measure}, \eqref{main2} can be written  as 
 \begin{align}
  \nonumber &\left(\int_\mathbb{G}\left[\exp\left(\frac{1}{|\mathbb{B}(0,|w|)|}\int_{\mathbb{B}(0,|w|)}\log\,f(z^{-\frac{1}{\epsilon}})\,dz\right)\right]^q \tilde{u}(w)\,dw\right)^\frac{1}{q}\le\\& \quad\quad\quad\quad\quad C \left(\int_\mathbb{G} f^p(w^{-\frac{1}{\epsilon}})\,\tilde{v}(w)\,dw\right)^\frac{1}{p}.
 \end{align}
 Again, applying Theorem \ref{Th2}, using polar decomposition, \eqref{measure} and some variable transformations
 we get the required result \eqref{DQ1}.
 \end{proof}
Next, we will define the conjugate of Theorem \ref{THM2} for the power weights  $u(x)=|\mathbb{B}(0,|x|)|^a$ and $v(x)=|\mathbb{B}(0,|x|)|^b$. We have the following inequality.
\begin{theorem} \label{thm3.5} Let $\mathbb{G}$ be a homogeneous group with the homogeneous dimension  $Q$ equipped with a quasi norm $|\cdot|.$
Let $0<p\le q<\infty,$ and let $a,b \in \mathbb{R}$ and $\epsilon>0.$ Then  for all positive functions $f$ on $\mathbb{G}$ the inequality 
\begin{align}\label{Thm4}
   \nonumber&\left(\int_\mathbb{G}\left[\exp\left(\epsilon|\mathbb{B}(0,|x|)|^\epsilon \int_{\mathbb{G}\backslash\mathbb{B}(0,|x|)}|\mathbb{B}(0,|y|)|^{(-\epsilon-1)}\log f(y)\,dy\right)\right]^q|\mathbb{B}(0,|x|)|^a dx\right)^\frac{1}{q}\le\\&\quad\quad\quad\quad\quad\quad\quad\quad\quad C\left(\int_\mathbb{G}f^p(x)|\mathbb{B}(0,|x|)|^b\,dx\right)^\frac{1}{p}
\end{align}
holds for some finite constant $C,$ if and only if 
\begin{align*}
  p\left(a+1\right) -q\left(b+1\right)=0.
\end{align*} Moreover, the least possible constant $C$ such that \eqref{Thm4} holds can be estimated as follows:
\begin{align*}
 \left(\frac{p}{q}\right)^\frac{1}{q} \epsilon^{\frac{1}{p}-\frac{1}{q}}\, \exp\bigg({-\left(\frac{b+1}{\epsilon p}\bigg)-\frac{1}{p}\right)}  \le C\le \left(\frac{p}{q}\right)^\frac{1}{q} \epsilon^{\frac{1}{p}-\frac{1}{q}}\, \exp\bigg({-\left(\frac{b+1}{\epsilon p}\right)}\bigg). 
\end{align*}
\end{theorem}
\begin{proof}
The  proof of this theorem 
follows exactly same lines as the proof of Theorem 
\ref{THM2} and therefore, we omit the proof.
\end{proof}
\begin{rem}
If we take $p=q$ and  so that $a=b,$ then the inequality \eqref{Thm4} holds with the constant $C=\exp\bigg(-{\frac{b+1}{\epsilon p}}\bigg).$ By using the test function in inequality \eqref{Thm4} 
$$f_\delta(x)=\begin{cases}\exp\bigg({\frac{b+1}{\epsilon p}}\bigg)|\mathbb{B}(0,1)|^{-(b+1)}|x|^{-\frac{Q}{p}(b+1-\epsilon \delta)},\quad x\in\mathbb{B}(0,1),\\\exp\bigg({\frac{b+1}{\epsilon p}}\bigg)|\mathbb{B}(0,1)|^{-(b+1)}|x|^{-\frac{Q}{p}(b+1+\epsilon \delta)},\quad x\in \mathbb{G}\backslash\mathbb{B}(0,1),\end{cases}$$ and suppose that $\delta\rightarrow 0^+,$ it can be shown that the constant $C=\exp\bigg(-{\frac{b+1}{\epsilon p}}\bigg)$ is sharp.
\end{rem}

\section*{acknowledgement} 

The completion of a portion of this paper was made possible through the support of Ghent Analysis \& PDE Centre at Ghent University, Belgium. The second author expresses sincere gratitude to the centre for their financial assistance and warm hospitality during her research visit. The research activities of MR are funded by the FWO Odysseus 1 grant G.0H94.18N: Analysis and Partial Differential Equations, the Methusalem programme of the Ghent University Special Research Fund (BOF) under grant number 01M01021, and the EPSRC grant EP/R003025/2. AS acknowledges the support of a UGC Non-NET fellowship from Banaras Hindu University, India.

\end{document}